\documentclass[preprint,12pt,authoryear]{elsarticle}

\usepackage{natbib}
\bibpunct{(}{)}{;}{a}{}{,} 
\usepackage{algorithm, algorithmicx, algpseudocode}
\usepackage{amsthm, amsmath, amssymb, mathrsfs, multirow, url, subfigure}
\usepackage{graphicx}
\usepackage{ifthen} 
\usepackage{amsfonts}
\usepackage{bbm}
\usepackage[shortlabels]{enumitem}
\newcommand*{\rom}[1]{\expandafter\@slowromancap\romannumeral #1@}
\usepackage[usenames]{color}
\usepackage{fullpage}
\usepackage{csquotes}
\theoremstyle{plain} 
\newtheorem{theorem}{Theorem}

\newtheorem{lemma}{Lemma}
\newtheorem{assumption}{Assumption}
\theoremstyle{definition}

\newtheorem{remark}{Remark}

\DeclareMathOperator*{\argmin}{arg\,min}

\journal{Journal of Statistical Planning and Inference}

\begin{document}

\begin{frontmatter}



\title{Bayesian nonparametric tests for multivariate locations}


\author{Indrabati Bhattacharya}

\address{Department of Biostatistics and Computational Biology\\
University of Rochester Medical Center}
\ead{indrabati_bhattacharya@urmc.rochester.edu}

\author{Subhashis Ghosal}
\ead{sghosal@ncsu.edu}

\address{Department of Statistics\\
North Carolina State University}

\begin{abstract}
In this paper, we propose novel, fully Bayesian non-parametric tests for one-sample and two-sample multivariate location problems. We model the underlying distribution using a Dirichlet process prior, and develop a testing procedure based on the posterior credible region for the spatial median functional of the distribution. For the one-sample problem, we fail to reject the null hypothesis if the credible set contains the null value. For the two-sample problem, we form a credible set for the difference of the spatial medians of the two samples and we fail to reject the null hypothesis of equality if the credible set contains zero. We derive the local asymptotic power of the tests under shrinking alternatives, and also present a simulation study to compare the finite-sample performance of our testing procedures with existing parametric and non-parametric tests.

\end{abstract}


\begin{keyword}
Bayesian nonparametrics \sep Hypothesis testing\sep credible region\sep Pitman alternatives.

\end{keyword}

\end{frontmatter}




%\begin{thebibliography}{00}


\section{Introduction}
\label{3S1}
Several frequentist testing procedures for multivariate locations are available in the literature, both parametric and non-parametric. The most well-known parametric procedure is the Hotelling's $T^2$-test, which is based on the multivariate mean vector and the covariance matrix, and it relies on the assumption of multivariate normality. This technique performs well if the assumption of multivariate normality is nearly correct, but suffers heavily otherwise, or in the presense of outliers. Non-parametric and robust alternatives based on signs and ranks have been quite popular over the years \citep{oja2004multivariate}. 

The notions of signs and ranks are based on the \enquote{ordering} of the data points, but in the multivariate setting, there is no objective basis of ordering. The notions are generalized to higher dimensions using $\ell_1$-objective functions (see Section \ref{3S2}). The existing one-sample location problem have the following set up. Suppose that, we have $n$ independent and identically distributed (i.i.d.) observations $Y=(Y_1,\dots,Y_n)$ taking values in $\mathbb{R}^k$ from a distribution $P(\cdot-\theta)$, a $k$-variate continuous distribution centered at $\theta=(\theta_1,\dots,\theta_k)^T$. Our objective is to test the hypothesis
\begin{equation}
H_0: \theta=\theta_0\quad \text{vs.}\quad H_1:\theta \neq \theta_0. 
\end{equation}
The existing score-based non-parametric test procedures are based on the multivariate spatial sign vector $U$, mulivariate spatial rank $R$, and multivariate spatial signed rank $Q$, which are defined respectively as
\begin{align}\label{e1}
    U(y)=& \begin{cases}
    \Vert y \Vert_2^{-1}y,\quad y\neq 0,\\
    0,\qquad\qquad y=0,
    \end{cases}\\
    \label{e2}
    R(y;Y)=\ &\frac{1}{n}\sum_{i=1}^nU(y-Y_i),\\
    \label{e3}
    Q(y;Y)=&\  \frac{1}{2}[R(y;Y)+R(y;-Y)].
\end{align}
The estimator of the location associated with spatial sign vector in \eqref{e1} is the spatial median
\begin{equation}
\hat{\theta}_n=\argmin_{\theta \in \mathbb{R}^k}\mathbb{P}_n\Vert Y-\theta\Vert_2,
\end{equation}
where $\mathbb{P}_n=n^{-1}\sum_{i=1}^n\delta_{Y_i}$ is the empirical measure. The score functions \eqref{e2} and \eqref{e3} give rise to multivariate Hodges-Lehmann estimators \citep{oja2004multivariate}. One drawback of these multivariate sign and rank-based tests is that their $p$-values rely on a limiting chi-square distribution of the test statistics, provided the underlying distribution is elliptically symmetric (defined in Section \ref{3S2}). In this paper, we construct Bayesian non-parametric testing procedures for multivariate locations using the spatial median. In other words, here we focus on the score functions of type \eqref{e1} and propose a non-parametric Bayesian testing procedure. Such a procedure is more attractive because it directly provides a credible set for the spatial median through quick posterior sampling, hence a testing criterion can be formulated without depending on asymptotics. We assume that the observations are drawn from a random distribution $P$, and we put a Dirichlet process (details given in Section \ref{3S3}) prior on it. From $P$, we can infer about its spatial median functional
\begin{equation}
    \theta(P)=\argmin_{\theta \in \mathbb{R}^k}P(\Vert Y-\theta \Vert_2-\Vert Y \Vert_2),
\end{equation}
where $Pf=\int f \mathrm{d}P$. The exact posterior distribution (modulo the Monte Carlo error) of $\theta(P)$ can be obtained easily by posterior simulation. Thus, we can form a credible region for $\theta(P)$ and our decision is based on whether the null value $\theta_0$ falls into this credible set. For elliptically symmetric distributions, this testing procedure effectively studies the one-sample location problem described above, but our testing procedure can be used to study a wider range of distributions $P$, where we study the null hypothesis $H_0:\theta(P)=0$. We show that our testing procedure is asymptotically non-parametric, i.e., the limiting type I error does not depend on the true distribution, and we further compute the asymptotic power function under Pitman (contiguous) alternatives along possible directions. The two-sample test has be formulated in a similar way, and its properties have been explored in a similar fashion. 

The development of the asymptotic theory for the testing procedures relies on a strengthening of the theory developed in \cite{bhattacharya2019bayesian}, which studies the asymptotic properties of a multivariate median, denoted by $\theta(P)$, in a non-parametric Bayesian framework. Precisely, they put a Dirichlet process prior on $P$ and proved a Bernstein-von Mises theorem for $\theta(P)$ (Theorem 3.1 of \cite{bhattacharya2019bayesian}), which we use for the derivation of the theorems in Sections \ref{3S3} and \ref{3S4}.

The rest of this paper is organized as follows. In Section \ref{3S2}, we give an overview of the existing multivariate testing procedures. In Section \ref{3S3}, we describe our proposed Bayesian non-parametric test procedures. Section \ref{3S4} gives the local asymptotic power under contiguous alternatives and Section \ref{3S5} presents a simulation study. All the proofs are given in Section \ref{3S6}, and we close the paper with a brief discussion in Section 7.

\section{Overview of existing tests}\label{3S2}
We begin this section by briefly describing the existing non-parametric testing procedures for one-sample location problems, and later move on to two-sample and several samples problems. Let $Y_1,\dots,Y_n$ be $n$ i.i.d. observations from a $k$-variate probability distribution $P$. According to \cite{sirkia2007multivariate}, the non-parametric testing methods can be classified as based on a multivariate spatial sign function $U$, a multivariate spatial rank $R$, and a multivariate spatial signed rank $Q$, which are defined as follows.

The test statistic based on the score function $T(Y)$, which is a general notation for the score functions described in Equations \eqref{e1}, \eqref{e2} and \eqref{e3}, is given by $n^{-1}\sum_{i=1}^nT(Y_i)$. Under $H_0$, $n^{-1/2}\sum_{i=1}^n T(Y_i) \rightsquigarrow \mathrm{N}_k(0, \Sigma)$, where $\Sigma=P\{T(Y)T(Y)^T\}$. The usual estimator for $\Sigma$ is $\hat{\Sigma}=n^{-1}\sum_{i=1}^nT(Y_i)T(Y_i)^T$. The appropriate cut-off for constructing the test procedure depends on the assumption of elliptical symmetry of $P$ \citep{oja2004multivariate}. The underlying distribution is said to be elliptically symmetric if its density is of the form
$$
f(y-\theta)=\vert \Sigma \vert ^{-1/2}g((y-\theta)^T\Sigma^{-1}(y-\theta)),
$$
with a symmetry center $\theta$, and a positive definite scatter matrix $\Sigma$. The univariate non-negative function $g(\cdot)$ satisfies the condition $\int_0^\infty u^{k/2-1}g(u)\mathrm{d}u<\infty$, so that $f$ is a valid density \citep{gomez2003survey}. The contours of these densities form concentric ellipses around the center $\theta$. Under $H_0$, 
$$
V^2=n\Big\Vert \hat{\Sigma}^{-1/2}\frac{1}{n}\sum_{i=1}^n T(Y_i)\Big\Vert^2 \rightsquigarrow \chi_k^2,
$$
where $\rightsquigarrow$ denotes convergence in distribution, and $\chi_k^2$ denotes a chi-square distribution with $k$ degrees of freedom \citep{sirkia2007multivariate}. Note that $V^2$ is $n$ times the squared length of the average standardized score vectors. For elliptically symmetric distributions, $V^2$ is strictly distribution free \citep{oja2004multivariate}. An approximate p-value can be obtained from the above limiting chi-square distribution. For small sample sizes, a conditional distribution-free p-value can be obtained under the assumption of directional symmetry (under which $(Y-\theta)/\Vert Y-\theta\Vert_2$ has the same distribution as $(\theta-Y)/\Vert \theta-Y\Vert_2$). This p-value can be obtained as $\mathrm{E}_\delta[\mathbbm{1}\{V_\delta^2 \geq V^2\}]$, where $\mathrm{E}_\delta$ is the expectation for the uniform distribution $\delta$ over the set of $2^k$ $k$-vectors with each component being $+1$ or $-1$, and $V_\delta^2$ is the value of the test statistic for the data points $\delta_1Y_1,\dots,\delta_nY_n$ \citep{oja2004multivariate}.

The one sample testing procedure has been naturally extended to two samples. Suppose that, we have two independent random samples $Y_{1}^{(j)},\dots,Y_{n_j}^{(j)}$, from $k$-variate distributions $P(\cdot-\theta^{(j)})$, $j=1,2$. We test the hypothesis
$$
H_0: \theta^{(1)}=\theta^{(2)},\quad \text{against}\ H_1: \theta^{(1)}\neq \theta^{(2)}.
$$
\cite{sirkia2007multivariate} developed a testing procedure using the general score function $T(Y)$ based on the following \textit{inner standardization} approach. First, a $k\times k$ matrix $H$ and a $k$-vector have to be found such that, for $Z_{i}^{(j)}=H(Y_{i}^{(j)}-h),\ i=1,\dots,n_j,\ j=1,2$,
\begin{align*}
    \frac{1}{n}\sum_{j=1}^2\sum_{i=1}^{n_j}T(Z_i^{(j)})=& 0,\\
    \frac{k}{n}\sum_{j=1}^2\sum_{i=1}^{n_j}T(Z_{i}^{(j)})T(Z_{i}^{(j)})^T=& \bigg\{\frac{1}{n}\sum_{j=1}^2\sum_{i=1}^{n_j}\Vert T(Z_i^{(j)})\Vert_2^2\bigg\}I_k,
\end{align*}
where $n=n_1+n_2$, and $I_k$ denotes the identity matrix of order $k\times k$. The test statistic has the form 
\begin{equation}\label{3eq7}
V^2= k\frac{\sum_{j=1}^2 n_j\Vert \frac{1}{n_j}\sum_{i=1}^{n_j}T(Z_{i}^{(j)})\Vert_2^2}{\frac{1}{n}\sum_{j=1}^2\sum_{i=1}^{n_j}\Vert T(Z_{i}^{(j)})\Vert_2^2}.
\end{equation}
It has been shown that $V^2$ has a limiting chi-square distribution with $k$ degrees of freedom. Thus, for large samples, a p-value can be constructed using the quantiles of the chi-square distribution. For smaller samples, an approximate p-value can be obtained using a conditionally distribution-free \textit{permutation test} version \citep{sirkia2007multivariate}, similar to $V_\delta$ in the one sample problem. This approach has been extended to a general $c$ number of samples as well.
\section{Bayesian Non-parametric Tests}
\label{3S3}
\subsection{One-sample Problem}
Suppose that, we have $n$ observations $Y_1,\dots,Y_n$ taking values in $\mathbb{R}^k$ from a $k$-dimensional distribution $P$. We choose a non-parametric Bayesian approach, i.e., we impose a prior on the underlying random distribution $P$, and form a credible set based on the posterior distribution of the spatial-median functional
\begin{equation}
    \theta(P)=\argmin_{\theta \in \mathbb{R}^k}P\{\Vert Y-\theta \Vert_2-\Vert Y \Vert_2\}.
\end{equation}
The hypothesis of interest is
\begin{equation*}
H_0: \theta(P)=\theta_0\quad \text{vs.}\quad H_1:\theta(P) \neq \theta_0. 
\end{equation*}
The most commonly used prior on $P$ is a Dirichlet process prior with centering measure $\beta$ ($\mathrm{DP}(\beta)$) (see Chapter 4, \cite{ghosal2017fundamentals}). A Dirichlet process prior can be alternatively denoted as $\mathrm{DP}(MG)$, where $M=\vert \beta \vert$, and $\bar{\beta}=\beta/M$ has cumulative distribution function $G$. The notations $\mathrm{DP}(\beta)$ and $\mathrm{DP}(MG)$ will be used interchangeably in this paper. Precisely, our chosen Bayesian model is given by
\begin{equation}\label{3eq8}
Y_1,\dots,Y_n\vert P\overset{iid}{\sim} P,\quad P\sim\mathrm{DP}(MG).
\end{equation}
The process $\mathrm{DP}(\beta)$ is a conjugate prior for i.i.d. observations from $P$, and the posterior distribution of $P$ given $Y_1,\dots,Y_n$ is $\mathrm{DP}(\beta+n\mathbb{P}_n)$. The exact posterior distribution of $\theta(P)$ cannot be obtained analytically, but posterior samples can be drawn via the stick-breaking construction of a Dirichlet process (Chapter 4, \cite{ghosal2017fundamentals}). If $\xi_1,\xi_2,\dots \overset{iid}{\sim} \bar{\beta}$, and $V_1,V_2,\dots \overset{iid}{\sim} \mathrm{Be}(1,M)$ are independent random variables and $W_j=V_j\prod_{l=1}^{j-1}(1-V_l)$, then $P=\sum_{j=1}^\infty W_j\delta_{\xi_j}\sim \mathrm{DP}(M\bar\beta)$. The posterior Dirichlet process can also be written in the form $\mathrm{DP}(M\bar{\beta})$ using the updating rule
$$
M\mapsto M+n,\quad \bar{\beta}\mapsto \frac{M}{M+n}\bar{\beta}+\frac{n}{M+n}\mathbb{P}_n.
$$
The non-informative limit as $M\to 0$ of the posterior of $P$, denoted by $\mathrm{DP}(n\mathbb{P}_n)$, is called the Bayesian bootstrap distribution. Its centering measure is $\mathbb{P}_n$, and a random distribution generated from it is supported on the observation points. It has the representation $P=\sum_{i=1}^nW_i\delta_{Y_i}$, where $W_i=U_i/\sum_{j=1}^n U_j$, with $U_1,\dots,U_n\overset{iid}{\sim}\mathrm{Exp}(1)$. If we choose the non-informative limit of the posterior Dirichlet process, we do not need to generate posterior samples from the full Dirichlet process, rather we only need to sample $n$ independently and identically distributed (i.i.d.) observations from an exponential distribution with parameter 1, which saves a lot of computational cost. Thus, a posterior $100(1-\alpha)\%$ credible region can be formed by the following steps.
\begin{itemize}
    \item For $b=1,\dots,B$, draw $U_{1b},\dots,U_{nb}\stackrel{iid}{\sim} \mathrm{Exp}(1)$. Thus, we calculate the Bayesian bootstrap weights as $W_{ib}=U_{ib}/\sum_{j=1}^n U_{jb}$, $i=1,\dots n$. 
    \item Draw posterior samples $\theta_b,\ b=1,\dots,B$, using the expression
    \begin{equation*}
        \theta_b=\argmin_{\theta}\sum_{i=1}^n W_{ib}\Vert Y_{ib}-\theta \Vert_2. 
    \end{equation*}
    \item Compute the posterior mean $\bar{\theta}=B^{-1} \sum_{b=1}^B \theta_b$ and the posterior covariance matrix $S=B^{-1}\sum_{b=1}^B(\theta_b-\bar{\theta})(\theta_b-\bar{\theta})^{\prime}$.
    \item A $100(1-\alpha)\%$ credible set for $\theta(P)$ is then constructed as
    \begin{equation*}
        C(Y_1,\dots,Y_n; \alpha)=\{\theta:(\theta-\bar{\theta})^{\prime}S^{-1}(\theta-\bar{\theta})\leq r_{1-\alpha}\},
    \end{equation*}
    where $r_{1-\alpha}$ is the $100(1-\alpha)$th percentile of $(\theta_b-\bar{\theta})^{\prime}S^{-1}(\theta_b-\bar{\theta}),\ b=1,\dots, B$.
    \item We reject $H_0$ if $\theta_0 \notin C(Y_1,\dots,Y_n;\alpha)$.
\end{itemize}
 The credible set considered here can be called \textit{modulo Monte Carlo error} because it is constructed using simulated draws and hence is subject to the Monte Carlo error. However, the Monte Carlo error can be controlled and made arbitrarily small. Next, we investigate the asymptotic properties of this testing procedure. For all the theorems discussed here, we make the following assumptions on the underlying true distributions. Below, $P^\star$ is the general notation for the underlying distributions.
\begin{assumption}
The distribution $P^\star$ has a density that is bounded on bounded subsets of $\mathbb{R}^k$.
\end{assumption}
\begin{assumption}
The spatial median of $P^\star$, i.e., $\theta^\star=\theta(P^\star)$ is unique.
\end{assumption}

\begin{theorem}{\label{3thm1}}
Suppose that, under the null hypothesis, $Y\sim P_{\theta_0}$, i.e., $\theta(P_{\theta_0})=\theta_0$, and $P_{\theta_0}$ satisfies Assumptions 1-2. Then, the one-sample Bayesian non-parametric test for $H_0:\theta(P)=\theta_0$ is asymptotically nonparametric, i.e., $$
P_{\theta_0}(\theta_0 \in C(Y_1,\dots,Y_n;\alpha)) \rightarrow 1-\alpha,
$$
as $n\to\infty$.
\end{theorem}
As we have already mentioned, the testing procedure has been constructed only using the posterior samples, without relying on any asymptotic properties. The proof of Theorem \ref{3thm1} is based on convergence properties of the posterior mean $(\bar{\theta})$ and the covariance matrix $(S)$ of the spatial median $\theta(P)$, for the Bayesian model \eqref{3eq8}. Let $P^\star$ be the true distribution of $Y$, and $\theta^\star\equiv\theta(P^\star)$. Also define
\begin{align}
    U_{\theta,P}=&P\bigg( \frac{(Y-\theta)(Y-\theta)^T}{\Vert Y-\theta \Vert_2^2}\bigg)\\
    V_{\theta,P}=&P\bigg\{ \frac{1}{\Vert Y-\theta \Vert_2}\bigg(I_k -\frac{(Y-\theta)(Y-\theta)^T}{\Vert Y-\theta \Vert_2^2}\bigg)\bigg\}.
\end{align}
Under Assumptions 1-2 on $P^\star$, the posterior distribution of the spatial median $\theta(P)$ can be approximated by a Gaussian distribution in the Bernstein-von Mises sense (Theorem 3.1, \cite{bhattacharya2019bayesian}), i.e., given $Y_1,\dots,Y_n$
$$
\sqrt{n}(\theta(P)-\hat{\theta}_n) \rightsquigarrow \mathrm{N}_k(0, V_{\theta^\star,P^\star}^{-1}U_{\theta^\star,P^\star}V_{\theta^\star,P^\star}^{-1}). 
$$
In Lemma \ref{3l1} (see below), we strengthen the conclusion of Theorem 3.1 in \cite{bhattacharya2019bayesian} to establish the convergence properties of the posterior mean and covariance matrix. 
\begin{lemma}\label{3l1}
Suppose that, the true distribution of $Y_1,\dots,Y_n \in \mathbb{R}^k$ is $P^\star$, and $P^\star$ satisfies Assumptions 1--2. Then, under the Bayesian model \eqref{3eq8}, the posterior mean $\bar{\theta}$ and the covariance matrix $S$ can be written as $\bar{\theta}=\hat{\theta}_n +o_{P^\star}(n^{-1/2})$ and $nS=V_{\theta^\star,P^\star}^{-1}U_{\theta^\star,P^\star}V_{\theta^\star,P^\star}^{-1}+o_{P^\star}(1)$, respectively, where $\hat{\theta}_n$ is the sample spatial median of $Y_1,\dots,Y_n$.
\end{lemma}
\subsection{Two Sample Problem}
The Bayesian non-parametric testing procedure for two-sample location problem can be constructed generalizing the one-sample procedure. Suppose that, we have $n_1$ i.i.d. observations $Y_1^{(1)},\dots,Y_{n_1}^{(1)} $ from a $k$-variate distribution $P^{(1)}$, and $n_2$ i.i.d. observations $Y_1^{(2)},\dots,Y_{n_2} ^{(2)}$ from another $k$-variate distribution $P^{(2)}$, independent of $P^{(1)}$. We want to test the hypothesis
$$
H_0: \theta(P^{(1)})-\theta(P^{(2)})=0\quad \text{against}\quad H_1: \theta(P^{(1)})-\theta(P^{(2)}) \neq 0. 
$$
As we have previously mentioned, if $P^{(1)}=P(\cdot-\theta^{(1)})$ and $P^{(2)}=P(\cdot-\theta^{(2)})$ are elliptically symmetric distributions, then this problem boils down to studying the two-sample location problem $H_0:\theta^{(1)}-\theta^{(2)}=0$ against $H_1:\theta^{(1)}-\theta^{(2)} \neq 0$. We put a $\mathrm{DP}(MG)$ prior on both $P^{(1)}$ and $P^{(2)}$, for some $M>0$ and $G$, i.e.,
\begin{equation}\label{3eq9}
    Y_1^{(j)},\dots,Y_{n_j}^{(j)}\vert P^{(j)}\overset{iid}{\sim} P^{(j)},\ P^{(j)}\sim \mathrm{DP}(MG),\ j=1,2.
\end{equation}
Thus $P^{(1)}$ and $P^{(2)}$ have stick-breaking representations $P^{(1)}=\sum_{m=1}^\infty W_m^{(1)}\delta_{\xi_m^{(1)}}$ and $P^{(2)} = \sum_{m=1}^\infty W_m^{(2)}\delta_{\xi_m^{(2)}}$,
respectively, where $W_m^{(1)}$, $m=1,2,\dots$, and $W_m^{(2)}$, $m=1,2,\dots$, are drawn from $\mathrm{Be}(1,M)$. Also, $\xi_m^{(1)},\ m=1,2,\dots$, and $\xi_m^{(2)}$, $m=1,2,\dots$, are i.i.d. samples from $G$. The posterior distribution of $P^{(j)}$ is $\mathrm{DP}(\beta+\sum_{i=1}^{n_j}\delta_{Y_i^{(j)}})$, $j\in\{1,2\}$. Like before, we consider the Bayesian bootstrap approximations of the posteriors of $P^{(j)}$, which can be written as $P^{(j)}=\sum_{m=1}^{n_j}W^{(j)}_m\delta_{Y^{(j)}_m}$, where $W_m^{(j)}=U_m^{(j)}/\sum_{l=1}^{n_j} U_l^{(j)}$ with $U_1^{(j)},\dots U_{n_j}^{(j)}\overset{iid}{\sim}\mathrm{Exp}(1)$, $j=1,2$. We construct a $100(1-\alpha)\%$ credible set for $\theta(P^{(1)})-\theta(P^{(2)})$ by the following steps.
\begin{itemize}
    \item For $b\in \{1,\dots,B\}$ and $j\in\{1,2\}$, draw $U^{(j)}_{1b},\dots,U^{(j)}_{n_jb}\overset{iid}{\sim}\mathrm{Exp}(1)$. Calculate the Bayesian bootstrap weights as $W_{ib}^{(j)}=U^{(j)}_{ib}/\sum_{l=1}^{n_j} U^{(j)}_{lb}$, $i=1,\dots,n_j,\ j\in\{1,2\}$. 
    \item Draw posterior samples $\theta_b^{(j)}$, using the expressions
    \begin{align*}
        \theta_b^{(j)}=&\argmin_{\theta}\sum_{i=1}^{n_j} W^{(j)}_{ib}\Vert Y_{ib}^{(j)}-\theta \Vert_2,\ b\in \{1,\dots,B\},\ j \in \{1,2\}.
    \end{align*}
    \item Compute posterior means $\bar{\theta}^{(j)}=B^{-1} \sum_{b=1}^B \theta_b^{(j)}$ and posterior covariance matrices $S^{(j)}=B^{-1}\sum_{b=1}^B(\theta_b^{(j)}-\bar{\theta}^{(j)})(\theta_b^{(j)}-\bar{\theta}^{(j)})^{\prime}$, for $j\in\{1,2\}$.
    \item A $100(1-\alpha)\%$ credible set for $\theta(P^{(1)})-\theta(P^{(2)})$ is then given by
    \begin{equation}
    \begin{split}
        C(Y_1^{(1)},\dots,Y_{n_1}^{(1)},Y_1^{(2)},\dots, Y_{n_2}^{(2)};\alpha)=\{\theta_1-\theta_2:(\theta_1-\theta_2-\bar{\theta}^{(1)}+\bar{\theta}^{(2)})^{\prime}\\
        (S^{(1)}+S^{(2)})^{-1}(\theta_1-\theta_2- \bar{\theta}^{(1)}+\bar{\theta}^{(2)})\leq r_{1-\alpha}\},
    \end{split}
    \end{equation}
    where $r_{1-\alpha}$ is the $100(1-\alpha)$th percentile of $$(\theta_b^{(1)}-\theta_b^{(2)}-\bar{\theta}^{(1)}+\bar{\theta}^{(2)})^{\prime}(S^{(1)}+S^{(2)})^{-1}(\theta_b^{(1)}-\theta_b^{(2)}-\bar{\theta}^{(1)}+\bar{\theta}^{(2)}),$$ for
    $b\in\{1,\dots, B\}$.
    \item We reject $H_0$ if $0 \notin C(Y_1^{(1)},\dots,Y_{n_1}^{(1)},Y_1^{(2)},\dots,Y_{n_2}^{(2)};\alpha)$.
\end{itemize}
The next theorem shows that the above test is asymptotically non-parametric, provided the underlying true distributions satisfy Assumptions 1-2.
\begin{theorem}{\label{3thm2}}
Suppose that, under the null hypothesis, $Y^{(1)}\sim P^{(1)}_{\theta_0}$ and $Y^{(2)}\sim P^{(2)}_{\theta_0}$ independently, such that $\theta(P_{\theta_0}^{(1)})=\theta(P_{\theta_0}^{(2)})=\theta_0$ for some fixed $\theta_0\in\mathbb{R}^k$. Let $P^{(1)}_{\theta_0}$ and $P^{(2)}_{\theta_0}$ satisfy Assumptions 1-2. Then the two-sample Bayesian non-parametric test is asymptotically non-parametric, i.e., for $0 < \alpha <1$
$$
P_{\theta_0}^{(1)}P_{\theta_0}^{(2)}(0 \in C(Y_1^{(1)},\dots,Y_{n_1}^{(1)},Y_1^{(2)},\dots,Y_{n_2}^{(2)};\alpha)) \rightarrow 1-\alpha,
$$
for $n_1$, $n_2$ such that $n_1\to\infty$, $n_2\to\infty$, $n_1/(n_1+n_2)\to\lambda$, and $n_2/(n_1+n_2)\to 1-\lambda$, for some fixed $0<\lambda<1$.
\end{theorem}
To prove Theorem \ref{3thm2}, we first need to investigate the asymptotic properties of the posterior distribution of $\theta(P^{(1)})-\theta(P^{(2)})$. The next lemma gives a Bernstein-von Mises theorem for the difference of spatial medians for two independent samples $Y_1^{(1)},\dots, Y_{n_1}^{(1)} \sim P^{(1)}$, and $Y_1^{(2)},\dots, Y_{n_2}^{(2)} \sim P^{(2)}$ under the Bayesian model \eqref{3eq9}. The asymptotic result follows almost immediately from Theorem 3.1 in \cite{bhattacharya2019bayesian}. Before stating Lemma \ref{3l2}, we introduce a few more notations. Suppose, ${P^\star}^{(j)}$ denotes the true distribution of $Y^{(j)}$, and, ${\theta^\star}^{(j)}=\theta({P^\star}^{(j)})$ is the spatial median of ${P^\star}^{(j)}$, $j=1,2$. Next, let $\hat{\theta}_{n_j}^{(j)}$ denote the sample spatial median constructed from $Y_1^{(j)},\dots,Y_{n_j}^{(j)}$, $j=1,2$. 
\begin{lemma}\label{3l2}
Let ${P^\star}^{(j)}$, $j=1,2$, satisfies Assumptions 1--2. Then under the Bayesian model \eqref{3eq9}, and $n_1$, $n_2$ such that $n_1\to\infty$, $n_2\to\infty$, $n_1/(n_1+n_2)\to\lambda$, and $n_2/(n_1+n_2)\to 1-\lambda$ for some fixed $0<\lambda<1$,
\begin{enumerate}[label={\upshape(\roman*)}]
\item $
\begin{aligned}[t]
\sqrt{n}(\hat{\theta}^{(1)}_{n_1}-{\theta^\star}^{(1)}-\hat{\theta}^{(2)}_{n_2}+{\theta^\star}^{(2)}) \rightsquigarrow &\mathrm{N}_k(0,\lambda^{-1}V^{-1}_{{\theta^\star}^{(1)},{P^\star}^{(1)}}U_{{\theta^\star}^{(1)},{P^\star}^{(1)}}V^{-1}_{{\theta^\star}^{(1)},{P^\star}^{(1)}}\\&+(1-\lambda)^{-1}V^{-1}_{{\theta^\star}^{(2)},{P^\star}^{(2)}}U_{{\theta^\star}^{(2)},{P^\star}^{(2)}}V^{-1}_{{\theta^\star}^{(2)},{P^\star}^{(2)}})
\end{aligned}
$,
\item given $Y_1^{(j)},\dots,Y_{n_j}^{(j)}$, $j=1,2$, 
\begin{align*}
\sqrt{n}(\theta(P^{(1)})-\hat{\theta}^{(1)}_{n_1}-\theta(P^{(2)})+\hat{\theta}^{(2)}_{n_2}) \rightsquigarrow &\mathrm{N}_k(0,\lambda^{-1}V^{-1}_{{\theta^\star}^{(1)},{P^\star}^{(1)}}\linebreak U_{{\theta^\star}^{(1)},{P^\star}^{(1)}}V^{-1}_{{\theta^\star}^{(1)},{P^\star}^{(1)}}\\&+
(1-\lambda)^{-1}V^{-1}_{{\theta^\star}^{(2)},{P^\star}^{(2)}}U_{{\theta^\star}^{(2)},{P^\star}^{(2)}}V^{-1}_{{\theta^\star}^{(2)},{P^\star}^{(2)}})
\end{align*}
in ${P^\star}^{(1)}\times {P^\star}^{(2)}$-probability.
\end{enumerate}
\end{lemma}
\section{Asymptotic power Under Contiguous Alternatives}
\label{3S4}
In this section, we analyze the local asymptotic power of the proposed Bayesian non-parametric tests, i.e., the limiting power under a sequence of alternatives converging to the null value. For the one-sample problem, we consider differentiable in quadratic mean (DQM) densities $\mathcal{P}=\{p_\theta=\mathrm{d}P_\theta/\mathrm{d}\mu: \theta \in \mathbb{R}^k\}$, i.e., there exists a vector valued measurable function $\dot{\ell}_\theta: \mathbb{R}^k\rightarrow \mathbb{R}^k$ such that, for $h\rightarrow 0$,
$$
\int \Big (\sqrt{p_{\theta+h}}-\sqrt{p_\theta}-\frac{1}{2}h^T\dot{\ell}_\theta\sqrt{p_\theta}\Big)^2\mathrm{d}\mu=o(\Vert h \Vert_2^2).
$$
We consider shrinking alternatives of the form 
\begin{equation}\label{3eq11}
H_{1n}: \theta=\theta_0+ \frac{h}{\sqrt{n}},
\end{equation}
which are also called Pitman alternatives, for models $P_\theta\in\mathcal{P}$. Here, we study the limiting power for the sequence of distributions $P_{\theta_0+h/\sqrt{n}} \in \mathcal{P}$. As a consequence of the DQM condition, the models $P_{\theta_0+h/\sqrt{n}}^n$ satisfy the local asymptotic normality (LAN) condition, i.e., there exist a matrix $I_\theta$ and a random vector $\Delta_{n,\theta}\rightsquigarrow \mathrm{N}_k(0,I_\theta)$ such that, for every converging sequence $h_n \rightarrow h$,
\begin{equation}\label{eq312}
\log\frac{\mathrm{d}P_{\theta+h_n/\sqrt{n}}^n}{\mathrm{d}P_{\theta}^n} =h^T \Delta_{n,\theta}-\frac{1}{2}h^T I_\theta h+o_{P_{\theta}^n}(1).
\end{equation}
In this context, specifically $\Delta_{n,\theta}=n^{-1/2}\sum_{i=1}^nh^T \dot{\ell}_\theta(Y_i)$, and $I_\theta= P_\theta \dot{\ell}_\theta \dot{\ell}_\theta^T$. The next theorem gives the limiting power for the one-sample test under a sequence of alternatives of the form $H_{1n}$ for the DQM models. 
\begin{theorem}\label{3thm3}
Suppose that, $P_{\theta_0}$ satisfies Assumptions 1--2. As $n\to\infty$, for a sequence of shrinking alternatives of the form \eqref{3eq11}, i.e., under a sequence of DQM models $P_{\theta_0+h/\sqrt{n}} \in \mathcal{P}$, the limiting power of the one-sample Bayesian non-parametric test for $H_0:\theta(P)=\theta_0$ is given by $F_{\chi^2}(\chi^2_{k;\alpha};k,\delta_1^\prime(V_{\theta_0,P_{\theta_0}}^{-1} U_{\theta_0,P_{\theta_0}} V_{\theta_0,P_{\theta_0}}^{-1})^{-1}\delta_1)$, where 
\begin{equation}
    \delta_1=P_{\theta_0}\bigg(-V_{\theta_0,P_{\theta_0}}^{-1} \frac{Y-\theta_0}{\Vert Y-\theta_0 \Vert_2}\{\dot{\ell}_{\theta_0}(Y)\}^T h\bigg),
\end{equation}
and $F_{\chi^2}(x;k,\delta)$ is the distribution function of a non-central chi-square distribution with $k$ degrees of freedom and non-centrality parameter $\delta$, with $\chi^2_{k;\alpha}$ being the $100(1-\alpha)$th percentile of the $\chi^2_k$ distribution.
\end{theorem}
For the two sample problem, we again consider DQM models $P^{(1)}_{\theta_0+h_1/\sqrt{n_1}}$, $P^{(2)}_{\theta_0+h_2/\sqrt{n_2}} \in \mathcal{P}$, i.e., the contiguous alternatives are of the form
\begin{equation}\label{eq316}
    H_{1n}:\theta_{n_j}^{(j)}=\theta_0+ \frac{h_j}{\sqrt{n_j}},\ j=1,2.
\end{equation}
The following theorem gives the limiting power of the two-sample test under contiguous alternatives of the form \eqref{eq316}, and the notations from Theorem \ref{3thm3} directly translate to the next theorem.
\begin{theorem}\label{3thm4}
Suppose that, $P^{(1)}_{\theta_0}$ and $P^{(2)}_{\theta_0}$ are mutually independent, and satisfy Assumptions 1--2. For a sequence of shrinking alternatives of the form \eqref{eq316}, i.e., for a sequence of DQM models $P^{(1)}_{\theta_0+h_1/\sqrt{n_1}},\ P^{(2)}_{\theta_0+h_2/\sqrt{n_2}}\in \mathcal{P}$, the asymptotic power of the two-sample Bayesian non-parametric test for testing $H_0: \theta(P^{(1)})=\theta(P^{(2)})=\theta_0$ for any $\theta_0 \in \mathbb{R}^k$, is given by 
$$
F_{\chi^2}(\chi^2_{k;\alpha}; \delta_2^\prime  (V_{\theta_0;P_{\theta_0}^{(1)}}^{-1}U_{\theta_0;P_{\theta_0}^{(1)}}V_{\theta_0,P_{\theta_0}^{(1)}}^{-1}+V_{\theta_0;P_{\theta_0}^{(2)}}^{-1}U_{\theta_0;P_{\theta_0}^{(2)}}V_{\theta_0,P_{\theta_0}^{(2)}}^{-1})^{-1}\delta_2),
$$where 
\begin{equation}
\begin{split}
    \delta_2=&\frac{1}{\sqrt{\lambda}}P_{\theta_0}^{(1)}\Big(- V_{\theta_0,P_{\theta_0}^{(1)}}^{-1} \frac{ Y^{(1)}-\theta_0 }{\Vert Y^{(1)}-\theta_0 \Vert_2} \{\dot{\ell_{\theta_0}^{(1)}}(Y^{(1)})\}^{T} h_1\Big) +\\
    &\frac{1}{\sqrt{1-\lambda}}P_{\theta_0}^{(2)}\Big(-V_{{\theta_0},P_{\theta_0}^{(2)}}^{-1} \frac{ Y^{(2)}-\theta_0 }{\Vert Y^{(2)}-\theta_0 \Vert_2}\{\dot{\ell_{\theta_0}^{(2)}}(Y^{(2)})\}^{T}h_2\Big),
    \end{split}
\end{equation}
for $n_1$, $n_2$ such that $n_1\to\infty$, $n_2\to\infty$, $n_1/(n_1+n_2)\to\lambda$, and $n_2/(n_1+n_2)\to 1-\lambda$, for some fixed $0<\lambda<1$.
\end{theorem}
\section{Simulation Study}
\label{3S5}
We conduct a simulation study to demonstrate the finite sample performance of the proposed one-sample and two-sample Bayesian non-parametric tests, for $k=2$ and $k=10$. The choices $k=2$ and $k=10$ will help us visualize the power of the tests in relatively lower and higher dimensional scenarios. We compare our tests with the Hotelling's $T^2$-test, and the spatial sign and rank tests, discussed in Section \ref{3S2}. Recall that, for testing $H_0:\theta=0$, we can denote the spatial sign and rank statistics by the general notation 
$$
V^2=n\Big\Vert \hat{\Sigma}^{-1/2}\frac{1}{n}\sum_{i=1}^n T(Y_i)\Big\Vert^2,
$$
with $T(y)=U(y)=y/\Vert y \Vert_2$ for the sign test, and $T(y)=R(y)=n^{-1}\sum_{i=1}^n U(y-Y_i)$, for the rank test. For the two-sample test, the sign and rank statistics are given by \eqref{3eq7}. The underlying distributions are $k$-variate Gaussian, $k$-variate $t$ with 1 degree of freedom (both elliptically symmetric), and $k$-variate gamma (asymmetric), for $k=2$ and $k=10$. For $k=2$, the covariance and the scale matrices for the normal and the $t$ distributions respectively, have been chosen to be the $k\times k$ identity matrix, denoted by $\mathbf{I}_k$. Since correlation structure plays an important role in higher dimensions, for $k=10$, we choose the covariance and scale matrix to be $\Sigma$ such that $\Sigma_{ij}=1$ for $i=j$, and 0.7 otherwise, for $i,j\in\{1,\dots,k\}$.

The $k$-variate gamma distribution is constructed using Gaussian copula \citep{xue2000multivariate}. To describe the construction briefly, let $Y_1,\dots,Y_k$ be $k$ many univariate gamma random variables $\mathrm{Ga}(s,r)$ with distribution functions and density functions being denoted by $F_j$ and $f_j$, $j=1,\dots,k$. Then the joint density of $Y=(Y_1,\dots,Y_k)$ is given by
$$
g(y,s,r,V)=c_\phi\{F_1,\dots,F_k \vert V \}\prod_{j=1}^k f_j(y_j,s,r),
$$
where $c_\phi(\cdot\vert V)$ denotes the density of the $k$-dimensional Gaussian copula. For $k=2$, we choose $s=(3,3)^T$, $r=(1,1)^T$, and $V$ such that $V_{11}=V_{22}=1$, and $V_{12}=0.5$. For $k=10$, we choose $s=r=\mathbf{1}_{10}$ and $V$ such that $V_{11}=V_{22}=1$, and $V_{12}=0.7$, where $\mathbf{c}_k$ denotes the vector of all $c$'s of length $k$. The usual Hotelling's $T^2$ statistic is based on the assumption of Gaussianity. Hence for comparison, here we choose a general version of Hotelling's $T^2$-test, where the Gaussian assumption can be relaxed to existence of second moments \citep{ito1956asymptotic}. For the general Hotelling's $T^2$-test, the p-value is based on a chi-square approximation instead of the usual $F$-distribution. For the one-sample test, we consider a sample size of $n=100$. For the two-sample test, we choose $n_1=100$, $n_2=90$ for $k=2$, and $n_1=100$, $n_2=60$ for $k=10$, to evaluate the performance of the tests where the data sizes are relatively unbalanced. The credible sets are constructed using 5000 posterior draws, and the power is calculated as the proportion of times the null hypotheses are rejected off 2000 replications. The location parameters are chosen suitably to show a good range of powers. 

Tables \ref{3tab1} and \ref{3tab2} show the power values for $k=2$, and it can be noted that our test procedures attain the nominal level $0.05$, and outperforms all other procedures in most scenarios. When the underlying distributions are not Gaussian, our method performs better than other methods, especially compared with the Hotelling's $T^2$ -test. Note that for the bivariate gamma distribution, the powers for the Hotelling's $T^2$-test in Table \ref{3tab1} are larger, which may lead us to believe that it performs better compared to the other procedures. However, table \ref{3tab1} shows that the corresponding sizes are also large.

Tables \ref{3tab3} and \ref{3tab4} demonstrate the power for $k=10$, and the Bayesian nonparametric tests perform relatively well here as well. However, the computation of the $10$-dimensional spatial median for each posterior distribution is somewhat expensive. Therefore, construction of the credible region takes a significantly longer time for the 10-dimensional scenario. Since the posterior contraction rate for $\theta(P)$ remains $n^{-1/2}$ for any finite dimension, the testing procedure does not suffer from the curse of dimensionality.
\begin{table}[t]
\begin{center}
    \begin{tabular}{c|cccc}
        \hline
        $\theta$ & NPBayes & Sign Test & Rank Test & Hotelling's $T^2$ \\
        \hline
        & \multicolumn{4}{c}{Bivariate Gaussian Distribution}\\
        \hline
        (0,\ 0) & 0.050 & 0.046 & 0.051 & 0.055\\
        (0.05,\ 0.05) & 0.139 & 0.086 & 0.084 & 0.099\\
        (0.1,\ 0.05) & 0.169 & 0.125 & 0.141 & 0.156\\
        (0.1,\ -0.1) & 0.221 & 0.188 & 0.213 & 0.234 \\
        \hline
        & \multicolumn{4}{c}{Bivariate $t_1$ Distribution}\\
        \hline
        (0,\ 0) & 0.054 & 0.053 & 0.041 & 0.020 \\
        (0.05,\ 0.05) & 0.174 & 0.058 & 0.053 & 0.025\\
        (0.1,\ 0.05) & 0.179 & 0.094 & 0.082 & 0.018 \\
        (0.1,\ -0.1) & 0.201 & 0.171 & 0.197 & 0.026\\
         \hline
        & \multicolumn{4}{c}{Bivariate Gamma Distribution}\\
        \hline
        (0,\ 0) & 0.049 & 0.016 & 0.025 & 0.294\\
        (0.05,\ 0.05) & 0.027 & 0.021 & 0.039 & 0.528\\
        (0.1,\ 0.05) & 0.029 & 0.013 & 0.058 & 0.607\\
        (0.1,\ -0.1) & 0.034 &  0.009 & 0.018 & 0.255 \\
        \hline
    \end{tabular}
    \end{center}
    \caption{Power for testing $H_0: \theta(P)=0$ for bivariate Gaussian, bivariate $t$ (with 1 degree of freedom), and bivariate gamma distributions with different location parameters ($\theta$), using the nonparametric Bayes (NPBayes) test, spatial sign test, spatial rank test and the Hotelling's $T^2$-test.}
    \label{3tab1}
\end{table}
\begin{table}[t]
\begin{center}
    \begin{tabular}{c|c|cccc}
        \hline
      $\theta^{(1)}$  & $\theta^{(2)}$ & NPBayes & Sign Test & Rank Test & Hotelling's $T^2$ \\
        \hline
        & &  \multicolumn{4}{c}{Bivariate Gaussian Distribution}\\
        \hline
        (0,\ 0) & (0,\ 0) & 0.050 & 0.057 & 0.051 & 0.037\\
        (0,\ 0) & (0.1,\ 0) & 0.135 & 0.091 & 0.085 & 0.083\\
        (0,\ 0) & (0.1,\ 0.1) & 0.225 & 0.098 & 0.122 & 0.136\\
        (0,\ 0) & (0, 0.3) & 0.402 & 0.337 & 0.346 & 0.146 \\
        \hline
        & & \multicolumn{4}{c}{Bivariate $t_1$ Distribution}\\
        \hline
        (0,\ 0) & (0,\ 0) & 0.059 & 0.041 & 0.052 & 0.011 \\
        (0,\ 0) & (0,1.\ 0) & 0.141 & 0.060 & 0.074 & 0.026\\
        (0,\ 0) & (0.1,\ 0.1) & 0.158 & 0.087 & 0.099 & 0.022 \\ (0,\ 0) &
        (0,\ 0.3) & 0.307 & 0.248 & 0.213 & 0.023 \\
        \hline
        & & \multicolumn{4}{c}{Bivariate Gamma Distribution}\\
        \hline
        (0,\ 0) & (0,\ 0) & 0.024 & 0.020 & 0.019 & 0.017 \\
        (0,\ 0) & (0.1,\ 0) & 0.020 & 0.019 & 0.018 & 0.025\\
        (0,\ 0) & (0.1,\ 0.1) & 0.030 & 0.015 & 0.014 & 0.030\\
        (0,\ 0) & (0,\ 0.3) & 0.033 & 0.018 & 0.023 & 0.028
    \end{tabular}
    \end{center}
    \caption{Power for testing $H_0: \theta^{(1)}=\theta^{(2)}$ for bivariate Gaussian, bivariate $t$ (with 1 degree of freedom), and bivariate gamma distributions with different location parameters ($\theta^{(1)}$ and $\theta^{(2)}$), using the nonparametric Bayes (NPBayes) test, spatial sign test, spatial rank test and the Hotelling's $T^2$-test.}
    \label{3tab2}
\end{table}
\begin{table}[t]
\begin{center}
    \begin{tabular}{c|cccc}
        \hline
        $\theta$ & NPBayes & Sign Test & Rank Test & Hotelling's $T^2$ \\
        \hline
        & \multicolumn{4}{c}{10-variate Gaussian Distribution}\\
        \hline
        $v_0$ & 0.051 & 0.039 & 0.089 & 0.115\\
        $v_1$ & 0.192 & 0.171 & 0.115 & 0.560\\
        $v_2$ & 0.284 & 0.271 & 0.283 & 0.330\\
        \hline
        & \multicolumn{4}{c}{10-variate $t_1$ Distribution}\\
        \hline
        $v_0$ & 0.062 & 0.029 & 0.048 & 0.048 \\
        $v_1$ & 0.137 & 0.112 & 0.114 & 0.072\\
        $v_2$ & 0.348 & 0.340 & 0.332 & 0.138 \\
         \hline
        & \multicolumn{4}{c}{10-variate Gamma Distribution}\\
        \hline
        $v_0$ & 0.071 & 0.062 & 0.101 & 0.173\\
        $v_1$ & 0.262 & 0.122 & 0.358 & 0.925\\
        $v_2$ & 0.093 & 0.064 & 0.091 & 0.219\\
        \hline
    \end{tabular}
    \end{center}
    \caption{Power for testing $H_0: \theta(P)=v_0$ for 10-variate Gaussian, 10-variate $t$ (with 1 degree of freedom), and 10-variate gamma distributions for different location parameters $\theta=v_0,v_1,v_2$, with $v_0=(0,0,0,0,0,0,0,0,0,0)^T$, $v_1=(0,0,0,0,0,0.1,0.1,0.1,0.1,0.1)^T$, and $v_2=(0,0,0,0,0,0.1,0.1,0.1,-0.1,-0.1)^T$, using the nonparametric Bayes (NPBayes) test, spatial sign test, spatial rank test and the Hotelling's $T^2$-test.}
    \label{3tab3}
\end{table}
\begin{table}[t]
\begin{center}
    \begin{tabular}{c|c|cccc}
        \hline
      $\theta^{(1)}$  & $\theta^{(2)}$ & NPBayes & Sign Test & Rank Test & Hotelling's $T^2$ \\
        \hline
        & &  \multicolumn{4}{c}{10-variate Gaussian Distribution}\\
        \hline
        $v_0$ & $v_0$ & 0.054 & 0.046 & 0.038 & 0.048\\
        $v_0$ & $v_1$ & 0.172 & 0.087 & 0.109 & 0.138\\
        $v_0$ & $v_2$ & 0.101 & 0.150 & 0.080 & 0.136\\
        \hline
        & & \multicolumn{4}{c}{10-variate $t_1$ Distribution}\\
        \hline
        $v_0$ & $v_0$ & 0.059 & 0.056 & 0.036 & 0.034 \\
        $v_0$ & $v_1$ & 0.071 & 0.076 & 0.092 & 0.045\\
        $v_0$ & $v_2$ & 0.074 & 0.075 & 0.072 & 0.039 \\  
        \hline
        & & \multicolumn{4}{c}{10-variate Gamma Distribution}\\
        \hline
        $v_0$ & $v_0$ & 0.051 & 0.036 & 0.023 & 0.057 \\
        $v_0$ & $v_1$ & 0.054 & 0.051 & 0.045 & 0.070\\
        $v_0$ & $v_2$ & 0.132 & 0.064 & 0.125 & 0.037\\
        \hline
    \end{tabular}
    \end{center}
    \caption{Power for testing $H_0: \theta^{(1)}=\theta^{(2)}$ for 10-variate Gaussian, 10-variate $t$ (with 1 degree of freedom), and 10-variate gamma distributions with different location parameters $v_0=(0,0,0,0,0,0,0,0,0,0)^T,\ v_1=(0,0,0,0,0,0.1,0.1,0.1,0.1,0.1)^T$ and $v_2=(0,0,0,0,0,0.1,0.1,0.1,-0.1,-0.1)^T$, using the nonparametric Bayes (NPBayes) test, spatial sign test, spatial rank test and the Hotelling's $T^2$-test.}
    \label{3tab4}
\end{table}
\begin{remark}
Here we have considered tests for multivariate locations based on spatial medians, but these tests can be constructed using multivariate $\ell_1$-medians (with $\ell_p$-norms) as well. For some fixed $p>1$, the $\ell_1$-median for a $k$-variate distribution $P$ can be defined as
$$
\theta_p(P)=\argmin_{\theta \in \mathbb{R}^k}P\{\Vert Y-\theta\Vert_p-\Vert \theta \Vert_p\}.
$$
Bernstein-von Mises theorems of $\ell_1$-medians are available in the literature \citep{bhattacharya2019bayesian}. Hence the expressions for local asymptotic powers under shrinking alternatives can be obtained using those theorems.
\end{remark}
\begin{remark}
One may argue that, the Hotelling's $T^2$-test is designed for testing the mean vector, and hence is unsuitable as a competing method for nonparametric tests based on medians. However, people use Hotelling's $T^2$-statistic as a testing procedure for the location of a distribution, under the assumption of normality, for which the center of symmetry is same as both mean and median. Naturally, this assumption gets violated for non-normal distributions, for which nonparametric tests are more suitable.
\end{remark}
\section{Proofs}
\label{3S6}
We start off this section with the proofs of Lemma \ref{3l1}, and $\ref{3l2}$, and then proceed with the proofs of the main theorems. 
\begin{proof}[Proof of Lemma \ref{3l1}]
 Define $\theta(\mathbb{B}_n)=\argmin_\theta\mathbb{B}_n\Vert Y-\theta\Vert_2$, where $\mathbb{B}_n=\mathrm{DP}(n\mathbb{P}_n)$ is the Bayesian bootstrap distribution. It has been shown in Lemma 1 in \cite{bhattacharya2019bayesian} that, asymptotically, $\theta(P)$ is a Bayesian bootstrapped analog of a $Z$-estimator, which implies that, asymptotically, the posterior distribution of $\theta(P)$ is the same as the conditional distribution of $\theta(\mathbb{B}_n)$. Thus, our problem boils down to showing the consistency of the first and second moments of the bootstrap $Z$-estimator $\theta(\mathbb{B}_n)$.

\cite{cheng2015moment} proved the consistency of the bootstrap moment estimators for the class of exchangeably weighted bootstrap (see Section 2.2, \cite{cheng2015moment}). The Bayesian bootstrap weights fall into the class of the exchangeable bootstrap weights, and we have to show that the $\ell_1$-criterion function $m_\theta(y)=-\Vert y-\theta \Vert_2+\Vert y \Vert_2$ satisfies the following two sufficient conditions. Let $\mathbb{G}_n=\sqrt{n}(\mathbb{P}_n-P^\star)$ denotes the empirical process and $\mathbb{G}_n^\star=\sqrt{n}(\mathbb{B}_n-\mathbb{P}_n)$ denotes the bootstrap empirical process. Suppose that the following conditions hold.
\begin{enumerate}
    \item Let $\Theta$ be the compact parameter space. For any $\theta \in \Theta$, 
    $$
    P^\star(m_\theta-m_{\theta^\star})\lesssim -\Vert \theta-\theta^\star \Vert_2^2.
    $$
    \item Define $N_\delta=\{m_\theta-m_{\theta_0}: \Vert \theta-\theta \Vert_2 \leq \delta\}$. We have to show
    \begin{align}
    \label{eq318}
        \big(\mathrm{E}_X\Vert \mathbb{G}_n \Vert _{N_\delta}^{p^\prime}\big)^{1/p^\prime} &\lesssim \delta\\
        \big(\mathrm{E}_{XW}\Vert \mathbb{G}_n^\star \Vert_{N_\delta}^{p^\prime}\big)^{1/p^\prime} &\lesssim \delta,
    \end{align}
    for some $p^\prime > 2$.
\end{enumerate}
Then the assertion in Lemma \ref{3l1} holds. First, we need to show that the parameter space can be restricted to a compact subset of $\mathbb{R}^k$ with high probability. In Lemma 2 of \cite{bhattacharya2019bayesian}, it has been shown that for some $0<\epsilon<1/4$ and $K>0$ such that $P(\Vert Y \Vert_2\leq K)>1-\epsilon$, given $Y_1,\dots, Y_n$, $\Vert\theta(\mathbb{B}_n)\Vert_2 \leq 3K$ with high ${P^\star} ^n$-probability, which implies that asymptotically, given $Y_1,\dots,Y_n$, $\Vert\theta(P)\Vert_2 \leq 3K$ with high ${P^\star}^n$-probability.

After fixing $K>0$, we choose $\Theta=\{\theta: \Vert \theta \Vert_2 \leq 3K\}$. Since $\Theta$ is compact, Condition 1 can be shown from a Taylor series expansion around $\theta^\star$,
\begin{equation}\label{eq320}
        P^\star m_\theta-P^\star m_{\theta^\star}=(\theta-\theta^\star)^\prime P^\star\dot{m}_{\theta^\star}+\frac{(\theta-\theta^\star)^\prime V_{\theta^\star,P^\star}(\theta-\theta^\star)}{2}+o(\Vert \theta-\theta^\star \Vert^2).
    \end{equation}
    Since $\theta^\star$ is the maximizer of $P^\star m_\theta$, $P^\star \dot{m}_{\theta^\star}$ vanishes. The matrix $V_{\theta^\star,P^\star}$ is negative definite, and hence, the second term in the right hand side of \eqref{eq320} is bounded above by $-c\Vert \theta-\theta^\star \Vert_2^2$, for a positive constant $c$.
    
    Before proving Condition 2, we introduce some notations. For any class of functions $\mathcal{A}$, and metric $\ell$, its $\epsilon$-bracketing number is denoted as $N_{[\,]}(\epsilon,\mathcal{A},\ell)$. The corresponding bracketing entropy integral is defined as
    $$
    J_{[\,]}(\epsilon,\mathcal{A},\ell)= \int_0^\delta \sqrt{1+\log N_{[\,]}(\epsilon,\mathcal{A},\ell)}\mathrm{d}\epsilon.
    $$
     Following \cite{cheng2015moment}, a simple sufficient condition for \eqref{eq318} is the following global Lipschitz condition
    \begin{align}\label{eq321}
       \vert m_\theta(x)-m_{\theta^\star}(x)\vert \leq & \Vert \theta-\theta^\star \Vert_2 ,
    \end{align}
    for any $\theta \in \Theta$, and
    \begin{align}\label{eq322}
        J_{[\,]}(1,N_\delta,L_2(P^\star))+ \Vert M \Vert_{\psi_{p^\prime }} < \infty,
    \end{align}
    for some $p^\prime >2$, where $\Vert \cdot \Vert_{\psi_p}$ is the Orlicz norm with respect to the convex function $\psi_p(t)=\exp {(t^p-1)}$. In our case, \eqref{eq321} holds by the triangle inequality $\vert m_\theta(y)-m_{\theta^\star}(y)\vert \leq \Vert \theta-\theta^\star \Vert_2$. Since $M(y)=1$ for every $y$, we just have to show that $J_{[\,]}(1,N_\delta, L_2(P^\star)) < \infty$.
    
    By Example 19.7 of \cite{van2000asymptotic}, since $\vert m_\theta(y)-m_{\theta^\prime}(y)\vert \leq \Vert \theta-\theta^\prime \Vert_2$, for every $\theta,\ \theta^\prime \in \Theta$, there exists a constant $K$ such that 
    $$
    N_{[\,]}(1,N_\delta,L_2(P^\star)) \leq \bigg(\frac{\mathrm{diam}\ \Theta} {\epsilon}\bigg)^k,\ \text{for every}\ 0<\epsilon<\mathrm{diam}\ \Theta.
    $$
    Then, the entropy is of the order $\log(1/\epsilon)$. By a change of variable, it can be shown that $J_{[\,]}(1,N_\delta, L_2(P^\star)) < \infty$.
\end{proof}
\begin{proof}[Proof of Lemma \ref{3l2}]
From Theorem 3.1 of \cite{bhattacharya2019bayesian}, for $j=1,2$,
\begin{enumerate}[label={\upshape(\roman*)}]
\item $\sqrt{n_j}(\hat{\theta}_{n_j}^{(j)}-{\theta^\star}^{(j)}) \rightsquigarrow \mathrm{N}_k(0, V_{{\theta^\star}^{(j)},{P^\star}^{(j)}}^{-1}U_{{\theta^\star}^{(j)},{P^\star}^{(j)}}V_{{\theta^\star}^{(j)},{P^\star}^{(j)}}^{-1})$,
\item Given $Y_1^{(j)},\dots,Y_{n_j}^{(j)}$,
\begin{align*}
\sqrt{n_j}(\theta(P^{(j)})-\hat{\theta}_{n_j}^{(j)}) \rightsquigarrow \mathrm{N}_k(0, V_{{\theta^\star}^{(j)},{P^\star}^{(j)}}^{-1}U_{{\theta^\star}^{(j)},{P^\star}^{(j)}}V_{{\theta^\star}^{(j)},{P^\star}^{(j)}}^{-1}),
\end{align*}
\end{enumerate}
in ${P^\star}^{(j)}$ probability. From the independence of the two samples, the conclusion follows.
\end{proof}
\begin{proof}[Proof of Theorem \ref{3thm1}]
The probability of accepting the null hypothesis under the null distribution is
\begin{align*}
P_{\theta_0}(\theta_0 \in C(Y_1,\dots,Y_n))=& P_{\theta_0}((\bar{\theta}-\theta_0)^\prime S^{-1}(\bar{\theta}-\theta_0) \leq r_{1-\alpha}).
\end{align*}
Using Lemma \ref{3l1} and Theorem 3.1 of \cite{bhattacharya2019bayesian}, 
\begin{equation}\label{eq3223}
n(\bar{\theta}-\theta_0)^T S^{-1}(\bar{\theta}-\theta_0)\rightsquigarrow \chi_k^2,
\end{equation}
which implies that, under $H_0$, $r_{1-\alpha}=\chi^2_{k;\alpha}+o_{P_{\theta_0}}(1)$. The weak convergence in \eqref{eq3223} uses the fact that if $X \sim \mathrm{N}_k(0,I_k)$, then $X^TX \sim \chi_k^2$. Next, again using Lemma \ref{3l1}, Theorem 3.1 of \cite{bhattacharya2019bayesian} and Slutsky's theorem,
\begin{align*}
P_{\theta_0}(\theta_0 & \in C(Y_1,\dots,Y_n))= P_{\theta_0}((\bar{\theta}-\theta_0)^\prime S^{-1}(\bar{\theta}-\theta_0) \leq r_{1-\alpha})\\
=& P_{\theta_0}((\hat{\theta}_n-\theta_0)^\prime (V_{\theta_0,P_{\theta_0}}^{-1}U_{\theta_0,P_{\theta_0}}V_{\theta_0,P_{\theta_0}}^{-1})^{-1}(\hat{\theta}_n-\theta_0) + o_{P_{\theta_0}}(1) \leq \ \chi^2_{k;\alpha}+o_{P_{\theta_0}}(1))\\
&\rightarrow  1-\alpha.
\end{align*}
\end{proof}
\begin{proof}[Proof of Theorem \ref{3thm2}]
The proof is similar to that of Theorem \ref{3thm1}. Using Lemma \ref{3l2}, under $H_0$, $\sqrt{n}(\hat{\theta}_{n_1}^{(1)}-\hat{\theta}_{n_2}^{(2)})$ converges to a Gaussian distribution with mean $0$ and covariance matrix $\lambda^{-1}V_{\theta_0,P_{\theta_0}^{(1)}}^{-1}U_{\theta_0,P_{\theta_0}^{(1)}}V_{\theta_0,P_{\theta_0}^{(1)}}^{-1}+(1-\lambda)^{-1}V_{\theta_0,P_{\theta_0}^{(2)}}^{-1}U_{\theta_0,P_{\theta_0}^{(2)}}V_{\theta_0,P_{\theta_0}^{(2)}}^{-1}$. Using Lemma \ref{3l1}, Lemma \ref{3l2} and Slutsky's theorem, 
$$
n(\bar{\theta}^{(1)}-\bar{\theta}^{(2)})^\prime (S^{(1)}+S^{(2)})^{-1}(\bar{\theta}^{(1)}-\bar{\theta}^{(2)}) \rightsquigarrow\chi_k^2,
$$
which implies that, under $H_0$, $r_{1-\alpha}=\chi^2_{k;\alpha}+o_{P_{\theta_0}^{(1)}}(1)+o_{P_{\theta_0}^{(2)}}(1)$. Next, using Lemma \ref{3l2}, under $H_0$,
\begin{align*}
    (& P_{\theta_0}^{(1)}\times  P_{\theta_0}^{(2)})[(\bar{\theta}^{(1)}-\bar{\theta}^{(2)})^\prime (S^{(1)}+S^{(2)})^{-1}(\bar{\theta}^{(1)}-\bar{\theta}^{(2)}) \leq r_{1-\alpha}]\\
    &= (P_{\theta_0}^{(1)}\times P_{\theta_0}^{(2)})(n(\hat{\theta}_{n_1}^{(1)}-\hat{\theta}_{n_2}^{(2)})^\prime (\frac{1}{\lambda}V_{\theta_0,P_{\theta_0}^{(1)}}^{-1}U_{\theta_0,P_{\theta_0}^{(1)}}V_{\theta_0,P_{\theta_0}^{(1)}}^{-1}\\
    &+\frac{1}{1-\lambda}V_{\theta_0,P_{\theta_0}^{(2)}}^{-1}U_{\theta_0,P_{\theta_0}^{(2)}}V_{\theta_0,P_{\theta_0}^{(2)}}^{-1})^{-1}(\hat{\theta}_{n_1}^{(1)}-\hat{\theta}_{n_2}^{(2)})
    +o_{P_{\theta_0}^{(1)}}(1)+o_{P_{\theta_0}^{(2)}}(1)\\
    &\leq\chi_{k;\alpha}^2+o_{P_{\theta_0}^{(1)}}(1)+o_{P_{\theta_0}^{(2)}}(1))
     \rightarrow 1-\alpha.
\end{align*}
\end{proof}
\begin{proof}[Proof of Theorem \ref{3thm3}]
It is well known that the models $P_{\theta_0}^n$ and $P_{\theta_0+h/\sqrt{n}}^n$ are mutually contiguous for DQM models (Example 6.5, \cite{van2000asymptotic}). Under $H_0$, using Theorem 5.23 of \cite{van2000asymptotic}, $\hat{\theta}_n$ can be written as
\begin{equation}\label{eq24}
\sqrt{n}(\hat{\theta}_n-\theta_0)=-\frac{1}{\sqrt{n}}\sum_{i=1}^n V_{\theta_0,P_{\theta_0}}^{-1}\frac{Y_i-\theta_0}{\Vert Y_i-\theta_0 \Vert_2}+ o_{P^n_{\theta_0}}(1).
\end{equation}
Let $L_n$ denote the log-likelihood ratio $L_n=\log(\mathrm{d}P_{\theta_0+h/\sqrt{n}}^n/\mathrm{d}P_{\theta_0}^n)$. Using Equation \eqref{eq312}, $L_n$ is written as
\begin{equation}\label{eq25}
    L_n= \frac{1}{\sqrt{n}}h^T\sum_{i=1}^n\dot{\ell}_{\theta_0}(Y_i)-\frac{1}{2}h^TI_{\theta_0}h+o_{P^n_{\theta_0}}(1).
\end{equation}
Putting together Equations \eqref{eq24} and \eqref{eq25}, we have
\begin{equation*}
 (\sqrt{n}(\hat{\theta}_n-\theta_0),L_n)) =\Big(-\frac{1}{\sqrt{n}}\sum_{i=1}^n V_{\theta_0,P_{\theta_0}}^{-1}\frac{Y_i-\theta_0}{\Vert Y_i-\theta_0 \Vert_2}, \frac{1}{\sqrt{n}}h^T\sum_{i=1}^n\dot{\ell}_{\theta_0}(Y_i)-\frac{1}{2}h^TI_{\theta_0}h\Big)+o_{P^n_{\theta_0}}(1).
\end{equation*}
By the central limit theorem, $(\sqrt{n}(\hat{\theta}_n-\theta_0), L_n)$ tends to a $(k+1)$-dimensional Gaussian distribution with mean zero and covariance 
$$
\delta_1=P_{\theta_0}\Big(-V_{\theta_0,P_{\theta_0}}^{-1} \frac{Y-\theta_0}{\Vert Y-\theta_0 \Vert_2}\{\dot{\ell}_{\theta_0}(Y)\}^T h\Big).
$$
Then by Le Cam's third lemma (Example 6.7, \cite{van2000asymptotic}), under $P_{\theta_0+h/\sqrt{n}}$, $\sqrt{n}(\hat{\theta}_n-\theta_0)$ converges weakly to a Gaussian distribution with mean $\delta_1$. Following the arguments used in Theorem \ref{3thm1}, the local asymptotic power of the test is given by
\begin{align*}
    P_{\theta_0+h/\sqrt{n}}&(\bar{\theta}-\theta_0)^\prime S^{-1}(\bar{\theta}-\theta_0) \leq  r_{1-\alpha}) =\\ &P_{\theta_0+h/\sqrt{n}}(n(\hat{\theta}_n-\theta_0)^\prime (V_{\theta_0,P_{\theta_0}}^{-1}U_{\theta_0,P_{\theta_0}}V_{\theta_0,P_{\theta_0}}^{-1})^{-1}(\hat{\theta}_n-\theta_0)+ o_{P_{\theta_0}}(1)
    &\leq \chi_{k;\alpha}^2+o_{P_{\theta_0}}(1)).
\end{align*}
Under $P_{\theta_0+h/\sqrt{n}}$, $n(\hat{\theta}_n-\theta_0)^\prime (V_{\theta_0,P_{\theta_0}}^{-1}U_{\theta_0,P_{\theta_0}}V_{\theta_0,P_{\theta_0}}^{-1})^{-1}(\hat{\theta}_n-\theta_0)$ tends to a chi-square distribution with the non-centrality parameter $\delta_1^\prime (V_{\theta_0,P_{\theta_0}}^{-1}U_{\theta_0,P_{\theta_0}}V_{\theta_0,P_{\theta_0}}^{-1})^{-1} \delta_1$, which gives us the asymptotic power given in the statement of Theorem \ref{3thm3}.
\end{proof}
\begin{proof}[Proof of Theorem \ref{3thm4}]
The proof proceeds along the lines of Theorem \ref{3thm3}. The models $\{{P_{\theta_0}^{(1)}}\times P_{\theta_0}^{(2)}\}$ and $\{P_{\theta_0+h_1/\sqrt{n_1}}^{(1)}\times P_{\theta_0+h_2/\sqrt{n_2}}^{(2)}\}$ are mutually contiguous. From Theorem 5.23 of \cite{van2000asymptotic}, sample spatial medians have the following linearizations,
\begin{align}
\label{eq324}
    \sqrt{n_1}(\hat{\theta}_{n_1}^{(1)}-\theta_0)=& -\frac{1}{\sqrt{n_1}}\sum_{i=1}^{n_1}V_{\theta_0,P_{\theta_0}^{(1)}}^{-1}\frac{Y_i^{(1)}-\theta_0}{\Vert Y_i^{(1)}-\theta_0 \Vert_2}+o_{P_{\theta_0}^{(1)}}(1),\\
    \label{eq325}
    \sqrt{n_2}(\hat{\theta}_{n_2}^{(2)}-\theta_0)=& -\frac{1}{\sqrt{n_2}}\sum_{i=1}^{n_2}V_{\theta_0,P_{\theta_0}^{(2)}}^{-1}\frac{Y_i^{(2)}-\theta_0}{\Vert Y_i^{(2)}-\theta_0 \Vert_2}+o_{P_{\theta_0}^{(2)}}(1).
\end{align}
Subtracting \eqref{eq325} from \eqref{eq324}, under $H_0$,
\begin{align*}
\sqrt{n}(\hat{\theta}_{n_1}^{(1)}-&\hat{\theta}_{n_2}^{(2)})=-\bigg\{\frac{1}{\sqrt{n_1\lambda}}\sum_{i=1}^{n_1}V_{\theta_0,P_{\theta_0}^{(1)}}^{-1}\frac{ Y_i^{(1)}-\theta_0}{\Vert Y_i^{(1)}-\theta_0 \Vert_2}\\
&-\frac{1}{\sqrt{n_2(1-\lambda)}}\sum_{i=1}^{n_2}V_{\theta_0,P_{\theta_0}^{(2)}}^{-1}\frac{ Y_i^{(2)}-\theta_0}{\Vert Y_i^{(2)}-\theta_0 \Vert_2}\bigg\}+
o_{P_{\theta_0}^{(1)}}(1)+o_{P_{\theta_0}^{(2)}}(1).
\end{align*}
Define the log-likelihood ratio $L_{n_1,n_2}^\prime=\log(\mathrm{d}P_{\theta_0+h_1/\sqrt{n_1}}^{(1)}\mathrm{d}P_{\theta_0+h_2/\sqrt{n_2}}^{(2)}/\mathrm{d}P_{\theta_0}^{(1)}\mathrm{d}P_{\theta_0}^{(2)})$, which looks like
\begin{align*}
L_{n_1,n_2}^\prime=\frac{1}{\sqrt{n_1}}h_1^T\sum_{i=1}^{n_1}\dot{\ell}_{\theta_0}^{(1)}(Y_i^{(1)})-\frac{1}{2}h_1^TI_{\theta_0}^{(1)}h_1+\frac{1}{\sqrt{n_2}}h_2^T\sum_{i=1}^{n_2}\dot{\ell}_{\theta_0}^{(2)}(Y_i^{(2)})-\frac{1}{2}h_2^TI_{\theta_0}^{(2)}h_2\\
+o_{P_{\theta_0}^{(1)}}(1)+ o_{P_{\theta_0}^{(2)}}(1).
\end{align*}
By central limit theorem, the joint distribution of $\sqrt{n}(\hat{\theta}_{n_1}^{(1)}-\hat{\theta}_{n_2}^{(2)})$ and $L_{n_1,n_2}^\prime$ tends to a $(k+1)$-dimensional Gaussian distribution with mean zero and covariance 
\begin{equation*}
\begin{split}
\delta_2=\frac{1}{\sqrt{\lambda}}P_{\theta_0}^{(1)}\Big(-V_{\theta_0,P_{\theta_0}^{(1)}}^{-1} \frac{ Y^{(1)}-\theta_0 }{\Vert Y^{(1)}-\theta_0 \Vert_2}\{\dot{\ell_{\theta_0}^{(1)}}(Y^{(1)})\}^{T} h_1\Big)+\\
\frac{1}{\sqrt{1-\lambda}}P_{\theta_0}^{(2)} \Big(-V_{\theta_0,P_{\theta_0}^{(2)}}^{-1}\displaystyle \frac{ Y^{(2)}-\theta_0}{\Vert Y^{(2)}-\theta_0 \Vert_2}\{\dot{\ell_{\theta_0}^{(2)}}(Y^{(2)}\}^{T} h_2\Big).
\end{split}
\end{equation*}
By Le Cam's third lemma, under $P_{\theta_0+h_1/\sqrt{n_1}}^{(1)} \times P_{\theta_0+h_2/\sqrt{n_2}}^{(2)}$, $\sqrt{n}(\hat{\theta}^{(1)}_{n_1}-\hat{\theta}_{n_2}^{(2)})$ converges weakly to a Gaussian distribution with mean $\delta_2$. Thus following the arguments used in Theorem \ref{3thm2}, the asymptotic power is given by
\begin{equation*}
\begin{split}
    P_{\theta_0+h_1/\sqrt{n_1}}^{(1)}& \times  P_{\theta_0+h_2/\sqrt{n_2}}^{(2)}\{(\bar{\theta}_{n_1}^{(1)}-\bar{\theta}_{n_2}^{(2)})^\prime (S^{(1)}+S^{(2)})^{-1}(\bar{\theta}_{n_1}^{(1)}-\bar{\theta}_{n_2}^{(2)}) \leq r_{1-\alpha}\}\\
    &=P_{\theta_0+h_1/\sqrt{n_1}}^{(1)}\times  P_{\theta_0+h_2/\sqrt{n_2}}^{(2)}\{n(\hat{\theta}_{n_1}^{(1)}-\hat{\theta}_{n_2}^{(2)})^\prime (\lambda^{-1}V_{\theta_0,P_{\theta_0}^{(1)}}^{-1}U_{\theta_0,P_{\theta_0}^{(1)}}V_{\theta_0,P_{\theta_0}^{(1)}}^{-1}\\
    &+(1-\lambda)^{-1}V_{\theta_0,P_{\theta_0}^{(2)}}^{-1}U_{\theta_0,P_{\theta_0}^{(2)}}V_{\theta_0,P_{\theta_0}^{(2)}}^{-1})^{-1}(\hat{\theta}_{n_1}^{-1}-\hat{\theta}_{n_2}^{(2)})+o_{P_{\theta_0}^{(1)}}(1)+ o_{P_{\theta_0}^{(2)}}(1)\\
    &\leq \chi^2_{k;\alpha}+o_{P_{\theta_0}^{(1)}}(1)+ o_{P_{\theta_0}^{(2)}}(1)\}.
\end{split}
\end{equation*}
Therefore under $P_{\theta_0+h_1/\sqrt{n_1}}^{(1)}\times P_{\theta_0+h_2/\sqrt{n_2}}^{(2)}$, 
\begin{align*}
n(\hat{\theta}_{n_1}^{(1)}-\hat{\theta}_{n_2}^{(2)})^\prime (\frac{1}{\lambda}V_{\theta_0,P_{\theta_0}^{(1)}}^{-1}U_{\theta_0,P_{\theta_0}^{(1)}}V_{\theta_0,P_{\theta_0}^{(1)}}^{-1}+\frac{1}{1-\lambda}V_{\theta_0,P_{\theta_0}^{(2)}}^{-1}U_{\theta_0,P_{\theta_0}^{(2)}}V_{\theta_0,P_{\theta_0}^{(2)}}^{-1})^{-1}(\hat{\theta}_{n_1}^{(1)}-\hat{\theta}_{n_2}^{(2)})
\end{align*}
tends to a non-central chi-square distribution with the non-centrality parameter $\delta_2^\prime (\lambda^{-1} V_{\theta_0,P_{\theta_0}^{(1)}}^{-1}\linebreak U_{\theta_0,P_{\theta_0}^{(1)}}V_{\theta_0,P_{\theta_0}^{(1)}}^{-1}+(1-\lambda)^{-1}V_{\theta_0,P_{\theta_0}^{(2)}}^{-1}U_{\theta_0,P_{\theta_0}^{(2)}}V_{\theta_0,P_{\theta_0}^{(2)}}^{-1})^{-1}\delta_2$, which gives the asymptotic power.
\end{proof}
\section{Discussion}
The nonparametric Bayesian tests constructed here are interesting alternatives to the classical tests for multivariate location, because they are easy to construct and do not require asymptotics to set their critical values. Also, unlike the classical tests, our methods do not need the assumption of elliptical symmetry to determine the rejection criterion.

Theorems 1 and 2 show that the tests are asymptotically nonparametric, which means that the type I errors do not depend on the true distribution for large samples. Also, Theorems 3 and 4 investigate the asymptotic power under shrinking alternatives, which are motivated by arguing that a testing procedure should be powerful at values close to the truth. Besides having useful asymptotic properties, our tests are computationally simple and efficient as well. The Bayesian bootstrap approximation to the posterior Dirichlet process is quite useful since the credible regions can be constructed using simple Monte Carlo methods. We have investigated the performance of the tests in both relatively lower and higher dimensional settings, for balanced and unbalanced data sizes (for the two-sample test), and for correlated data, and the tests exhibit reasonable power in all scenarios.
\ifthenelse{1=1}{
\bibliographystyle{apalike}
\bibliography{sample.bib}

\begin{thebibliography}{}

\bibitem[Bhattacharya and Ghosal, 2020]{bhattacharya2019bayesian}
Bhattacharya, I. and Ghosal, S. (2020).
\newblock Bayesian inference on multivariate medians and quantiles.
\newblock {\em Statistica Sinica}.
\newblock To appear,
  \url{http://www3.stat.sinica.edu.tw/ss_newpaper/SS-2020-0108_na.pdf}.

\bibitem[Cheng, 2015]{cheng2015moment}
Cheng, G. (2015).
\newblock Moment consistency of the exchangeably weighted bootstrap for
  semiparametric {M}-estimation.
\newblock {\em Scandinavian Journal of Statistics}, 42(3):665--684.

\bibitem[Ghosal and van~der Vaart, 2017]{ghosal2017fundamentals}
Ghosal, S. and van~der Vaart, A. (2017).
\newblock {\em Fundamentals of {N}onparametric {B}ayesian {I}nference},
  volume~44.
\newblock Cambridge University Press.

\bibitem[G{\'o}mez S{\'a}nchez-Manzano et~al., 2003]{gomez2003survey}
G{\'o}mez S{\'a}nchez-Manzano, E., G{\'o}mez~Villegas, M.~A., and
  Mar{\'\i}n~Diazaraque, J.~M. (2003).
\newblock A survey on continuous elliptical vector distributions.
\newblock {\em Revista matem{\'a}tica complutense}, 16(1):345--361.

\bibitem[Ito, 1956]{ito1956asymptotic}
Ito, K. (1956).
\newblock Asymptotic formulae for the distribution of {H}otelling's generalized
  ${T}^2$-statistic.
\newblock {\em The Annals of Mathematical Statistics}, pages 1091--1105.

\bibitem[Oja and Randles, 2004]{oja2004multivariate}
Oja, H. and Randles, R.~H. (2004).
\newblock Multivariate nonparametric tests.
\newblock {\em Statistical Science}, 19(4):598--605.

\bibitem[Sirki{\"a} et~al., 2007]{sirkia2007multivariate}
Sirki{\"a}, S., Taskinen, S., Nevalainen, J., and Oja, H. (2007).
\newblock Multivariate nonparametrical methods based on spatial signs and
  ranks: The {R} package spatial{NP}.

\bibitem[Van~der Vaart, 1998]{van2000asymptotic}
Van~der Vaart, A.~W. (1998).
\newblock {\em Asymptotic Statistics}, volume~3.
\newblock Cambridge University Press.

\bibitem[Xue-Kun~Song, 2000]{xue2000multivariate}
Xue-Kun~Song, P. (2000).
\newblock Multivariate dispersion models generated from {G}aussian copula.
\newblock {\em Scandinavian Journal of Statistics}, 27(2):305--320.

\end{thebibliography}
}{

}
%\end{thebibliography}
\end{document}